\documentclass[12pt]{article}

\usepackage[english]{babel}

\usepackage{enumerate}

\usepackage{fullpage} 
\usepackage{parskip} 

\usepackage{setspace}
\usepackage{amsmath, amssymb, amsfonts, latexsym}
\usepackage{amsthm}
\usepackage{tikz,tikzsymbols}
\usepackage{mathtools}

\newtheorem{theorem}{Theorem}
\newtheorem*{theorem*}{Theorem}

\newtheorem*{thm*}{Theorem}

\newtheorem{cor}[theorem]{Corollary}
\newtheorem{lemma}[theorem]{Lemma}

\newtheorem*{conj*}{Conjecture}
\theoremstyle{definition}
\newtheorem{definition}[theorem]{Definition}

\numberwithin{theorem}{section}
\numberwithin{equation}{section}
\numberwithin{figure}{section}

\usepackage{hyperref}

\hypersetup{
  colorlinks   = true,
  urlcolor     = blue, 
  linkcolor    = blue,
  citecolor   = red
}

\newcommand{\J}{J}
\newcommand{\Jmax}{\J_{\max}}

\newcommand{\newG}{G^*}
\newcommand{\newH}{H^*}
\newcommand{\newJ}{J^*}

\newcommand{\newnewG}{G^{**}}

\newcommand{\newnewJ}{J^{**}}

\newcommand{\newpi}{\pi^*}

\newcommand{\newdeg}{\deg^*}

\newcommand{\tree}[2]{\mathcal{T}_{#1,#2}}

\newcommand{\broom}[2]{B_{#1,#2}}
\newcommand{\double}[2]{D_{#1,#2}}
\newcommand{\Tnd}{\tree{n}{d}}
\newcommand{\Tn}{\mathcal{T}_n}

\newcommand{\Anr}{\mathcal{A}_{n,r}}

\newcommand{\Bnd}{B_{n,d}}
\newcommand{\Bnr}{B_{n,r}}
\newcommand{\Dnd}{D_{n,d}}

\def\Tmeet{T_{\mathsf{meet}}}
\def\Tbestmeet{T_{\mathsf{bestmeet}}}

\DeclareMathOperator {\Ret}{Ret}

\title{Random Walks and the Meeting Time for Trees}
\author{Andrew Beveridge\footnote{Department of Mathematics, Statistics and Computer Science, Macalester College,  Saint Paul, MN, USA, \texttt{abeverid@macalester.edu}},
 Ben Bridenbaugh\footnote{Department of Mathematics, Statistics and Computer Science, Macalester College,  Saint Paul, MN, USA, \texttt{bbridenb@macalester.edu}}  \,
 and Ari Holcombe Pomerance\footnote{Department of Mathematical and Statistical Sciences, University of Colorado Denver, Denver CO, USA, \texttt{ari.holcombepomerance@ucdenver.edu}}
 }
\date{}

\begin{document}
\maketitle

\begin{abstract}
Consider a random walk on a tree $G=(V,E)$. For $v,w \in V$, let the hitting time $H(v,w)$ denote the expected number of steps required for the random walk started at $v$ to reach $w$, and let $\pi_v = \deg(v)/2|E|$ denote the stationary distribution for the random walk. 
We characterize the extremal tree structures for the meeting time
$\Tmeet(G) = \max_{w \in V} \sum_{v \in V} \pi_v H(v,w)$. For fixed order $n$ and diameter $d$, the meeting time is maximized by the broom graph. The meeting time is minimized by the balanced double broom graph, or a slight variant, depending on the relative parities of $n$ and $d$.
\end{abstract}

\section{Introduction}

Let $G= (V,E)$ be a connected graph.  A \emph{random walk} on $G$ starting at vertex $w$ is a sequence of vertices $w=w_0, w_1, \ldots, w_t, \ldots$ such that  for $t \geq 0$, we have  
$$
\Pr(w_{t+1} = v \mid w_t = u ) =
\left\{ 
\begin{array}{cl}
1/\deg(u) & \mbox{if } (u,v) \in E, \\
 0  & \mbox{otherwise}.
 \end{array}
 \right.
 $$
For an introduction to random walks on graphs, see  H\"aggstr\"om \cite{Haggstrom2002} and Lov\'asz \cite{Lovasz1996}. For surveys on contemporary random walk applications, see Masuda et al.~\cite{MPL2017} and Riascos and Mateos \cite{RM2021}. 

The \emph{hitting time} $H(u,v)$ is the expected number of steps before a random walk started at vertex $u$ hits vertex $v$, where we define $H(u,u)=0$ for the case $u=v$. When $G$ is not bipartite, 
 the distribution of $w_t$ converges to the \emph{stationary distribution} $\pi$, given by $\pi_v = \deg(v)/2|E|$. We have convergence in the bipartite case when we follow a \emph{lazy walk}, which remains at the current state with probability $1/2$. This simply doubles the hitting times, so we will consider non-lazy walks for simplicity.

Herein, we consider extremal questions for trees of order $n$ with diameter $d$. 

 \begin{definition}
 The family of trees of order $n$  is denoted by $\Tn$.
     The family of trees of order $n$ with diameter $d$ is denoted by $\tree{n}{d}$.
 \end{definition}

 Ciardo et al.~\cite{CDK2020} characterized the upper and lower bounds on Kemeny's constant 
 $$ \kappa(G) = \sum_{u \in V} \sum_{v \in V} \pi_u \pi_v H(u,v)$$ 
 for trees $G \in \tree{n}{d}$.  
The unique minimizer is the balanced lever graph: a path on vertices $v_0, v_1, \ldots, v_d$ with $n-d-1$ additional leaves adjacent to vertex $v_{\lfloor d/2 \rfloor}$. The unique maximizer is the balanced double broom: a path on vertices $v_1, v_2 \ldots, v_{d-1}$ with $\lfloor n-d+1 \rfloor$ leaves adjacent to $v_1$ and $\lceil n-d+1 \rceil$ leaves adjacent to $v_{d-1}$.
 We consider a variant of $\kappa(G)$ where we choose the worst target vertex rather than drawing the target vertex from the stationary distribution $\pi$.

\begin{definition}
For a graph $G=(V,E)$ and a vertex $w \in V$, the \emph{meeting time at} $w$ is 
$$
H(\pi, w) = \sum_{v \in V} \pi_v H(v,w).
$$
The \emph{meeting time of graph} $G$ is 
$$
\Tmeet(G) = \max_{w \in V} H(\pi, w).     
$$
\end{definition}
We determine upper and lower bounds on $\Tmeet(G)$ for $G \in \Tnd$, showing that there is a unique maximizer and a unique minimizer. 
We encounter three subfamilies of trees, and provide some examples in Figure \ref{fig:tree-examples}.
We start with the upper bound.

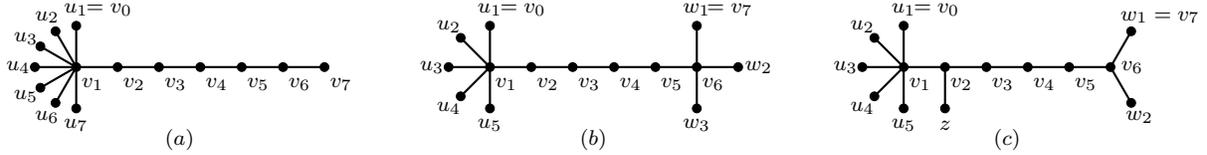
\begin{figure}

\begin{center}
   \begin{tikzpicture}[scale=0.55]

\begin{scope}

\foreach \x in {1,2,3,4,5,6,7}
{
\draw[thick] (0,0) -- (60+\x*180/6:1);
\draw[fill] (60+\x*180/6:1) circle (3pt);
\node at (60+\x*180/6:1.4) {\scriptsize $u_{\x}$};
}

\node at (.8,1.4) {\scriptsize $= v_0$};

\draw[thick] (0,0) -- (6,0);

\foreach \x in {0,1,2,3,4,5,6}
{
\draw[fill] (\x,0) circle (3pt);
}

\foreach \x in {1,2,3,4,5,6,7}
{
\node[below] at (\x - .6,0) {\scriptsize $v_{\x}$};
}

\node at (2.5,-1.75) {\scriptsize $(a)$};

\end{scope}

\begin{scope}[shift={(10,0)}]

\foreach \x in {1,2,3,4,5}
{
\draw[thick] (0,0) -- (\x*180/4 + 45:1);
\draw[fill] (\x*180/4 + 45:1) circle (3pt);
\node at (\x*180/4 + 45:1.4) {\scriptsize $u_{\x}$};
}

\node at (.8,1.4) {\scriptsize $= v_0$};

\begin{scope}[shift={(5,0)}]
\foreach \x in {1,2,3}
{
\draw[thick] (0,0) -- (180-\x*180/2:1);
\draw[fill] (180-\x*180/2:1) circle (3pt);
}

\node at (90:1.4) {\scriptsize $w_{1}$};
\node at (0:1.5) {\scriptsize $w_{2}$};
\node at (-90:1.4) {\scriptsize $w_{3}$};

\node at (.8,1.4) {\scriptsize $= v_7$};

\end{scope}

\draw[thick] (0,0) -- (5,0);

\foreach \x in {0,1,2,3,4,5}
{
\draw[fill] (\x,0) circle (3pt);
}

\foreach \x in {1,2,3,4,5,6}
{
\node[below] at (\x - .6,0) {\scriptsize $v_{\x}$};
}

\node at (2.5,-1.75) {\scriptsize $(b)$};

\end{scope}

\begin{scope}[shift={(20,0)}]

\foreach \x in {1,2,3,4,5}
{
\draw[thick] (0,0) -- (\x*180/4 + 45:1);
\draw[fill] (\x*180/4 + 45:1) circle (3pt);
\node at (\x*180/4 + 45:1.4) {\scriptsize $u_{\x}$};
}

\node at (.8,1.4) {\scriptsize $= v_0$};
\node at (1,-1.4) {\scriptsize $z$};

\begin{scope}[shift={(5,0)}]
\foreach \x in {1,2}
{
\draw[thick] (0,0) -- (180-\x*360/3:1);
\draw[fill] (180-\x*360/3:1) circle (3pt);
}

\begin{scope}[shift={(0.5,0)}]
\node at (60:1.4) {\scriptsize $w_{1} = v_7$};    
\end{scope}

\node at (-60:1.4) {\scriptsize $w_{2}$};

\end{scope}

\draw[thick] (0,0) -- (5,0);
\draw[thick] (1,0) -- (1,-1);

\foreach \x in {1,2,3,4,5}
{
\node[below] at (\x - .6,0) {\scriptsize $v_{\x}$};
}

\node[right] at (5,0) {\scriptsize $v_{6}$};

\foreach \x in {0,1,2,3,4,5}
{
\draw[fill] (\x,0) circle (3pt);
}

\draw[fill] (1,-1) circle (3pt);

\node at (2.5,-1.75) {\scriptsize $(c)$};

\end{scope}

\end{tikzpicture} 
\end{center}
    \caption{Some trees in $\tree{14}{7}$. (a) The broom graph $B_{14,7}$. (b) A double broom. (c) A near double broom.}
    \label{fig:tree-examples}
\end{figure}

\begin{definition}
\label{def:broom}
For $3 \leq d <n$, the broom $B_{n,d} \in \tree{n}{d}$  consists of a path $v_1, \ldots, v_{d}$ with  leaves $u_1, \ldots u_{n-d}$ incident with $v_1$, where we also label $v_0=u_1$ for convenience. The path $v_1, \ldots, v_{d}$ is the \emph{handle} of the broom, while $u_1, \ldots, u_{n-d}$ are the \emph{bristles} of the broom.     
\end{definition}

\begin{theorem}
\label{thm:max-meet}
For $3 \leq d < n$, the quantity 
$
\max_{G \in \mathcal{T}_{n,d}} \Tmeet(G) 
$ 
is achieved uniquely by the broom graph $B_{n, d}$. This value is
$$
\Tmeet(\Bnd) = H(\pi, v_d)
=
2(d-1) n + \frac{2d^3-6d^2+4d}{3(n-1)} -2 d^2+2 d+\frac{1}{2}.
$$
\end{theorem}

The lower bound is slightly more complicated. We must treat $d=3$ as a special case, and for $n \geq 4$, our minimizing structure depends on the parity of $n-d$.

\begin{definition}
\label{def:double-broom}
 For $3 \leq d < n$, a \emph{double broom} $G \in \tree{n}{d}$  consists of a path $v_1, \ldots, v_{d-1}$ with leaves $u_1, \ldots, u_{\ell}$ incident with $v_1$, and leaves
$w_1, \ldots, w_r$ incident with $v_{d-1}$. For convenience, we also label $v_0=u_1$ and $v_d=w_1$.
 Note that such double brooms satisfy $n=\ell+r+d-1$.
We define $\Dnd \in \Tnd$ to be the \emph{balanced double broom} that satisfies $\ell = \lfloor (n-d+1)/2 \rfloor$ and  $r = \lceil(n-d+1)/2 \rceil$.
\end{definition}

For example, $\tree{n}{3}$ is the set of double brooms with diameter 3. These are more intuitively called \emph{double stars}: the union of two star graphs whose centers are connected by an edge.  

\begin{theorem}
\label{thm:min-meet-diam-3}
For $n \geq 4$, the quantity 
$
\min_{G \in \mathcal{T}_{n,3}} \Tmeet(G) 
$ 
is achieved uniquely by the balanced double star $\double{n}{3}$.
When $n$ is even, we have
$$
\Tmeet(\double{n}{3}) = H(\pi, v_0) = \frac{5}{2}n - 4 - \frac{1}{2(n-1)},
$$
and when $n$ is odd, we have
$$
\Tmeet(\double{n}{3}) = H(\pi, v_0) = \frac{5}{2}n - 3.
$$
\end{theorem}

When $d \geq 4$, we need a second subfamily of trees in order to state our minimization result.

\begin{definition}
A \emph{near double broom} $G$ is a tree that is not a double broom, but with one leaf $z$ whose removal results in a double broom. This leaf $z$ is called the \emph{singleton leaf} of $G$. 
We define $\Dnd' \in \tree{n}{d}$ to be the \emph{balanced near double broom} with $\ell = \lfloor (n-d)/2 \rfloor$ and  $r = \lceil(n-d)/2 \rceil$ and whose singleton leaf $z$ is adjacent to $v_{\lfloor d/2 \rfloor}$.
\end{definition}

\begin{theorem}
\label{thm:min-meet}
For $4 \leq d < n$, the quantity 
$
\min_{G \in \mathcal{T}_{n,d}} \Tmeet(G) = 
\min_{G \in \mathcal{T}_{n,d}} \max_{v \in V} H(\pi, v)
$ 
is achieved uniquely by 
\begin{enumerate}[(a)]
\item the balanced double broom graph $\Dnd$ when $n$ and $d$ have opposite parities, with value
\begin{equation}
\label{eqn:min-meet-opposite-parity}
 \Tmeet(\Dnd) = H(\pi, v_d)
= \frac{1}{2}(d+2) n
+ \frac{d^3-6 d^2+8 d }{6(n-1)} 
-\frac{1}{2}(d+5).  
\end{equation}

\item the balanced near double broom graph $\Dnd'$ when $n$ and $d$ have the same parity, with value
\begin{equation}
\label{eqn:min-meet-same-parity}
\Tmeet(\Dnd') = 
\frac{1}{2}\left(d+2\right)
+ \frac{d^3-3d^2-d+6}{6(n-1)}
-
\begin{cases}
  \frac{1}{2}(d+5) & \mbox{if $d$ is even}, \\
 \frac{1}{2}(d+3)
 & \mbox{if $d$ is odd}. 
\end{cases}
\end{equation}
\end{enumerate}
\end{theorem}

The extremal graphs for $\tree{8}{k}$ and $2 \leq k \leq 7$ are shown in Figure \ref{fig:overview}.  Consolidating our results,
we obtain bounds on the meeting time of trees of order $n$. 

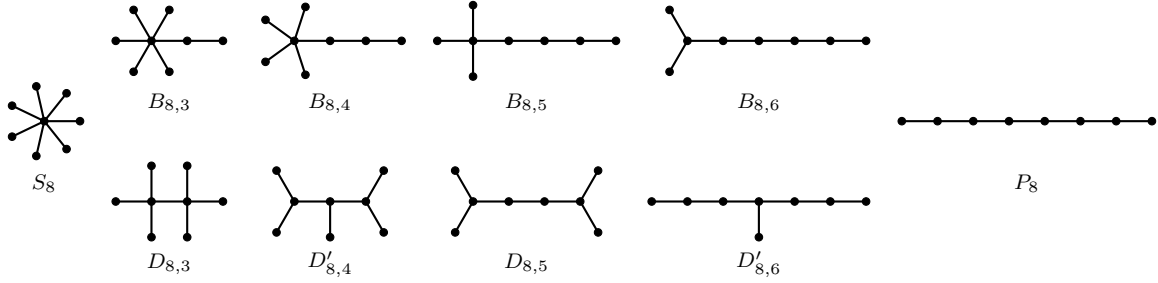
\begin{figure}

\begin{center}
   \begin{tikzpicture}[scale=0.475]

\begin{scope}

\foreach \x in {0,1,2,3,4,5,6}
{
\draw[thick] (0,0) -- (\x*360/7:1);
\draw[fill] (\x*360/7:1) circle (3pt);
}

\draw[fill] (0,0) circle (3pt);

\node at (0,-1.75) {\scriptsize $S_8$};

\end{scope}


\begin{scope}[shift={(3,2.25)}]

\foreach \x in {1,2,3,4,5}
{
\draw[thick] (0,0) -- (\x*360/6:1);
\draw[fill] (\x*360/6:1) circle (3pt);
}

\draw[thick] (0,0) -- (2,0);

\foreach \x in {0,1,2}
\draw[fill] (\x,0) circle (3pt);

\node at (.5,-1.75) {\scriptsize $B_{8,3}$};

\end{scope}

\begin{scope}[shift={(3,-2.25)}]

\foreach \x in {1,2,3}
{
\draw[thick] (0,0) -- (\x*360/4:1);
\draw[fill] (\x*360/4:1) circle (3pt);
}

\begin{scope}[shift={(1,0)}]
\foreach \x in {1,2,3}
{
\draw[thick] (0,0) -- (180-\x*360/4:1);
\draw[fill] (180-\x*360/4:1) circle (3pt);
}
\end{scope}

\draw[thick] (0,0) -- (1,0);

\foreach \x in {0,1}
\draw[fill] (\x,0) circle (3pt);

\node at (.5,-1.75) {\scriptsize $D_{8,3}$};

\end{scope}


\begin{scope}[shift={(7,2.25)}]

\foreach \x in {1,2,3,4}
{
\draw[thick] (0,0) -- (\x*360/5:1);
\draw[fill] (\x*360/5:1) circle (3pt);
}

\draw[thick] (0,0) -- (3,0);

\foreach \x in {0,1,2,3}
{
\draw[fill] (\x,0) circle (3pt);
}

\node at (1,-1.75) {\scriptsize $B_{8,4}$};

\end{scope}

\begin{scope}[shift={(7,-2.25)}]

\foreach \x in {1,2}
{
\draw[thick] (0,0) -- (\x*360/3:1);
\draw[fill] (\x*360/3:1) circle (3pt);
}

\begin{scope}[shift={(2,0)}]
\foreach \x in {1,2}
{
\draw[thick] (0,0) -- (180-\x*360/3:1);
\draw[fill] (180-\x*360/3:1) circle (3pt);
}
\end{scope}

\draw[thick] (0,0) -- (2,0);
\draw[thick] (1,0) -- (1,-1);

\foreach \x in {0,1,2}
{
\draw[fill] (\x,0) circle (3pt);
}

\draw[fill] (1,-1) circle (3pt);

\node at (1,-1.75) {\scriptsize $D_{8,4}'$};

\end{scope}


\begin{scope}[shift={(12,2.25)}]

\foreach \x in {1,2,3}
{
\draw[thick] (0,0) -- (\x*360/4:1);
\draw[fill] (\x*360/4:1) circle (3pt);
}

\draw[thick] (0,0) -- (4,0);

\foreach \x in {0,1,2,3,4}
{
\draw[fill] (\x,0) circle (3pt);
}

\node at (1.5,-1.75) {\scriptsize $B_{8,5}$};

\end{scope}

\begin{scope}[shift={(12,-2.25)}]

\foreach \x in {1,2}
{
\draw[thick] (0,0) -- (\x*360/3:1);
\draw[fill] (\x*360/3:1) circle (3pt);
}

\begin{scope}[shift={(3,0)}]
\foreach \x in {1,2}
{
\draw[thick] (0,0) -- (180-\x*360/3:1);
\draw[fill] (180-\x*360/3:1) circle (3pt);
}
\end{scope}

\draw[thick] (0,0) -- (3,0);

\foreach \x in {0,1,2,3}
{
\draw[fill] (\x,0) circle (3pt);
}

\node at (1.5,-1.75) {\scriptsize $D_{8,5}$};

\end{scope}


\begin{scope}[shift={(18,2.25)}]

\foreach \x in {1,2}
{
\draw[thick] (0,0) -- (\x*360/3:1);
\draw[fill] (\x*360/3:1) circle (3pt);
}

\draw[thick] (0,0) -- (5,0);

\foreach \x in {0,1,2,3,4,5}
{
\draw[fill] (\x,0) circle (3pt);
}

\node at (2,-1.75) {\scriptsize $B_{8,6}$};

\end{scope}

\begin{scope}[shift={(18,-2.25)}]

\draw[thick] (-1,0) -- (5,0);
\draw[thick] (2,0) -- (2,-1);

\foreach \x in {-1,0,1,2,3,4,5}
{
\draw[fill] (\x,0) circle (3pt);
}

\draw[fill] (2,-1) circle (3pt);

\node at (2,-1.75) {\scriptsize $D_{8,6}'$};

\end{scope}


\begin{scope}[shift={(25,0)}]

\draw[thick] (-1,0) -- (6,0);

\foreach \x in {-1,0,1,2,3,4,5,6}
{
\draw[fill] (\x,0) circle (3pt);
}

\node at (2.5,-1.75) {\scriptsize $P_8$};

\end{scope}
    
\end{tikzpicture} 
\end{center}
    \caption{The meeting time extremal graphs for each of $\mathcal{T}_{8,d}$ for $2 \leq d \leq 7$. The upper row contains the maximizers and the bottom row contains the minimizers.}
    \label{fig:overview}
\end{figure}

\begin{theorem}
\label{thm:max-trees}
The extremal meeting times for $G \in \Tn$ are as follows. 
We have
$$
\min_{G \in \Tn} \Tmeet(G) = \Tmeet(S_n) = 2n - \frac{7}{2}.
$$
and
$$
\max_{G \in \Tn} \Tmeet(G) = \Tmeet(P_n) = \frac{4n^2-8n+3}{6}.
$$
Furthermore, the star $S_n$ and the path $P_n$ are the unique extremal trees for the meeting time.
\end{theorem}

\section{Preliminaries}

We present a formula for the hitting time $H(u,v)$ for a tree $G=(V,E)$, and introduce the joining time, which is the meeting time scaled by $2|E|$. 

\subsection{Hitting times for trees}

For a vertex subset $W \subseteq V$, we define
$$
\deg(W) = \sum_{w \in W} \deg(w).
$$
The \emph{return time} $\Ret(u)$ is the expected number of steps required for a random walk started at $u$ to first return to $u$. It is well-known (see Lov\'asz \cite{Lovasz1996}) that
\begin{equation}
\label{eqn:return-time}
\Ret(u) = \frac{2|E|}{\deg(u)} = \frac{\deg(V)}{\deg(u)}.
\end{equation}
Suppose that $(u,v) \in E$ is an edge in the tree.   Removing this edge breaks $G$ into two disjoint trees $G_1$ and $G_2$, where $u \in V(G_1)$ and $v \in V(G_2)$. We define $V_{u:v} = V(G_1)$ and $V_{v:u} = V(G_2)$. Equivalently, we have
\begin{equation}
\label{eqn:adjacent-set}
  V_{u:v} = \{w \in V: d(w,u) < d(w,v) \}
  \quad \mbox{and} \quad
  V_{v:u} = \{w \in V: d(w,v) < d(w,u) \}.
\end{equation}
For these adjacent vertices $u$ and $v$, let $F$ denote the induced tree on $V_{u:v} \cup \{   v \}$. We have
\begin{equation} 
\label{eqn:hit-adjacent}
		H(u,v)= \Ret_{F} (v) - H_F(v,u) =   \Ret_{F} (v) - 1 
        = \deg(V_{u:v})= 2|V_{u:v}|-1,
\end{equation}
where the first equality holds by equation \eqref{eqn:return-time}.
In particular, when $z$ is a leaf and $y$ is its unique neighbor, we have
\begin{equation}
\label{eqn:hit-neighbor-to-leaf}
 H(y,z) 
 = \deg(V_{y:z}) = \deg(V \backslash \{ z \}) 
 = 2|V \backslash \{ z \}|-1 
 = 2|E|-1.   
\end{equation}
A similar argument produces a formula for the hitting times between any two vertices, but first we need some additional notation. Let  $u,v,w \in V$. Recall that $d(u,v)$ is the distance between these two vertices, and define
$$
\ell(u,v;w)=\frac{1}{2}(d(u,w)+d(v,w)-d(u,v))
$$
to be the length of the intersection of the $(u,w)$-path and the $(v,w)$-path. 

\begin{lemma}[Equation (2.6) of \cite{beveridge2009}]
\label{lemma:hit-time}
For any pair of vertices $u,v$, we have
\begin{equation} 
\label{eqn:hit-time}
H(u,v)=\sum_{w\in V}{\ell(u,w;v)\deg(w)}.
\end{equation}
\end{lemma}

We note that equivalent variations on this tree hitting time formula are common in the literature (c.f.~\cite{DTW2003,GW2013},  and Chapter 5 of \cite{aldous+fill}).
Finally, formulas for the meeting times of the path and the star were calculated in \cite{BW2013}.

\begin{theorem}[Theorem 1.2 of \cite{BW2013}]
\label{thm:meet-path-star}
For the path $P_n$ on vertices $v_0, v_1, \dots, v_{n-1}$, we have
$$
\Tmeet(P_n) = H(\pi, v_{n-1}) = \frac{4n^2-8n+3}{6} = \frac{2(n-1)^2}{3} - \frac{1}{6}.
$$
For the star $S_n$ with center $v_0$ and leaf vertices $v_1, \ldots, v_{n-1}$, we have
$$
\Tmeet(S_n) = H(\pi,v_{n-1}) = 2n - \frac{7}{2}.
$$
\end{theorem}

\subsection{The joining time}

Throughout this work, it will be convenient to scale meeting times by $2|E|=2(n-1)$. This clears the denominator but does not further impact our arguments, since we will always fix the order of the graphs under consideration. With that in mind, we make the following definitions.

\begin{definition}
Let $G=(V,E)$ be a graph. For a vertex $v\in V$,
the \emph{joining time to} $v$ is
\begin{equation}
\label{eqn:scaled-pi-to-vertex}
J(v) = 2|E| H(\pi,v)  = \sum_{u \in V} \deg(u) H(u,v).
\end{equation}
The \emph{maximum joining time of $G$} is
$$
\Jmax(G) = \max_{v \in V} J(v) = 2|E| \, \Tmeet(G).
$$
Given a subset $S \subseteq V$, and a vertex $v \in V$ the \emph{joining time from $S$ to $v$} is
\begin{equation}
\label{eqn:join-set-to-vertex}
J(S,v) = 2|E| \sum_{s \in S} \pi_s H(s,v) = \sum_{s \in S} \deg(s) H(s,v).
\end{equation}
\end{definition}

Sometimes it is convenient to combine equations \eqref{eqn:hit-time} and \eqref{eqn:scaled-pi-to-vertex} and write
\begin{equation}
\label{eqn:join-to-vertex-2}   
J(w) = \sum_{u \in V} \sum_{v \in V} \ell(u,v;w) \deg(u) \deg(v).
\end{equation}

\section{Maximizing the meeting time}

In this section, we prove Theorem \ref{thm:max-meet}, that $\max_{G \in \Tnd} \Tmeet(G)$ is achieved uniquely by the broom graph $\Bnd$.
We start with three useful results about joining times. First we relate the joining time of a leaf with the meeting time of its neighbor.

\begin{lemma}
    \label{lemma:unif-to-leaf}
    For any tree $G$ on $n \geq 2$ vertices, and leaf $z \in V$ with unique neighbor $y \in V$, we have
$$
J(z) 
= J(y) + 4(n^2-3n+2).
$$
\end{lemma}

\begin{proof}
By equation \eqref{eqn:hit-neighbor-to-leaf} and $|E|=n-1$, we have
\begin{align*}
\J(z) 
&= 
\J(y) - \deg(z) H(z,y) + \sum_{u \in V - z} \deg(u) H(y,z) \\
&=
\J(y) - 1 + (2|E|-1)^2 
= 
\J(y) + 4 |E|^2 - 4|E| = \J(y) + 4(n^2-3n+2),
\end{align*}
as desired.
\end{proof}

Next, we compare the joining time of adjacent vertices. 

\begin{lemma}\label{lemma:adjacent-join-times}
    Let $G\in \Tnd$ be a tree with adjacent vertices $u,v$.
    Then  $J(u)>J(v)$ if and only if $\deg(V_{u:v}) < \deg(V_{v:u})$.
\end{lemma}

\begin{proof}
By equations \eqref{eqn:scaled-pi-to-vertex} and \eqref{eqn:hit-adjacent}, we have 
$$
J(u) = J(V_{u:v},u) + J(V_{v:u},v) + \deg(V_{v:u}) H(v,u) 
= J(V_{u:v},u) + J(V_{v:u},v) + \deg(V_{v:u})^2 
$$
while
$$
J(v) = J(V_{v:u},v) + J(V_{u:v},u) + \deg(V_{u:v}) H(u,v) 
= J(V_{v:u},v) + J(V_{u:v},u) + \deg(V_{u:v})^2.
$$
Therefore $J(u) > J(v)$ if and only if $\deg(V_{v:u}) > \deg(V_{u:v})$.
\end{proof}

Finally, we can show that a leaf always maximizes the joining time.

\begin{cor} \label{cor:leaf-is-max}
    For any tree $G\in \Tnd$, the vertex $z$ that attains $\max_{v\in V} J(v)$ is a leaf.
\end{cor}
\begin{proof}
    We prove the contrapositive. Suppose that $\deg(z) = k \geq 2.$ Let $\{y_1,y_2,\ldots,y_k\}\subseteq V$ be the set of neighbors of $z$. 
Without loss of generality, let $\deg(V_{y_1:z})=\min_i \deg(V_{y_i:z})$. Since $\sum_i \deg(V_{y_i:z})=\deg(V)-\deg(z)$, we know that $\deg(V_{y_1:z}) < \deg(V)/2$. Therefore $\deg(V_{y_1:z}) < \deg(V_{z:y_1})$, and hence  $J(z)<J(y_1)$ by Lemma \ref{lemma:adjacent-join-times}.
\end{proof}

\subsection{The meeting time for a broom}

Recall that $\Bnd$ is the broom on $n$ vertices with diameter $d$. Its handle consists of vertices $v_1, v_2, \ldots, v_d$ and its bristles $u_1, \ldots, u_{n-d}$ are adjacent to $v_1$. For convenience, we also set $v_0 = u_1$. We calculate $\Jmax(\Bnd)$ and $\Tmeet(\Bnd)$.

\begin{lemma}
\label{lemma:broom-join-max}
Consider the broom graph $\Bnd$, where $d \geq 2$.
Then
\begin{equation}
\label{eqn:join-broom-tip}    
\Jmax(\Bnd) = J(v_d) =  
4 (d-1) n^2  +(5  -4 d^2) n+ \frac{4 d^3 -4 d-3}{3}.
\end{equation}
\end{lemma}

\begin{proof}
    Let $G = B_{n, d}$ with handle $v_1, \ldots, v_d$ and bristles
  $v_0=u_{1}, u_{2}, \ldots, u_{n-d}$.
By Corollary \ref{cor:leaf-is-max}, we know that $\Jmax(\Bnd)$ is achieved at a leaf. It is clear from equation \eqref{eqn:join-to-vertex-2} and the structure of the broom that the maximizing target vertex is the broom tip $v_d$.
  
    Using formula \eqref{eqn:hit-adjacent}, we calculate the hitting times between adjacent vertices in our graph:
\begin{align*}
H(v_0, v_1) = H(u_i,v_1) &=  1  &\mbox{for } 1 \leq i \leq n-d, \\
H(v_j, v_{j+1}) &= 2(n-d) + 2(j-1) + 1  &\mbox{for } 1 \leq j \leq d-1.
\end{align*}
Therefore
\begin{align*}
H(v_0, v_k) &= \sum_{j=0}^{k-1} H(v_j, v_{j+1})
= k + 2(k-1)(n-d) + 2 \sum_{j=1}^{k-1} (j-1) \\
&=
k + 2(k-1)(n-d) + (k-2)(k-1).
\end{align*}
In particular,
$$
H(v_0,v_d) = 2 d n-2 n -d^2+2,
$$
and for $1 \leq k \leq d-1$, we have
$$
H(v_k, v_d) = H(v_0,v_d) - H(v_0,v_k) = (d - k) (2n+k-d-2).
$$
 We use these formulas to find that
    \begin{align*}
        \J(v_d) &= \sum_{v_k \in V} \deg(v_k) H(v_k, v_d) \\
         &= 
         (n-d)  H(v_0, v_d) + (n - d + 1) H(v_1, v_d)+ 2\sum_{k=1}^{d-1} H(v_k, v_d)  \\
         &= \begin{multlined}[t]
          (n-d) (2 d n-2 n -d^2+2) 
         + (n - d + 1) (d - 1) (2n-d-1) \\
         + 2\sum_{k=2}^{d-1} (d - k) (2n+k-d-2)
         \end{multlined}\\
         &= 4 (d-1) n^2  +(5  -4 d^2) n+ \frac{4 d^3 -4 d-3}{3}
    \end{align*}
    where the final simplification can be validated via mathematical software like Mathematica or Sage.
 \end{proof}

As one would expect, brooms with larger handles (and the same order) have larger joining times.

\begin{cor}
\label{cor:max-join-increasing}
For a fixed order $n$, the maximum joining time
$\Jmax(\Bnd)$ is strictly increasing for $2 \leq d \leq n-1$.
\end{cor}

\begin{proof}
The derivative, with respect to $d$, of equation \eqref{eqn:join-broom-tip} is
$$
4 d^2-8 d n+4 n^2 - \frac{4}{3}
$$
which has critical points
$n \pm \sqrt{3}/3$ and is positive for $2 \leq d \leq n-1$.
\end{proof}

We conclude this subsection with the formula for the meeting time of the broom.

\begin{cor}
\label{cor:broom-pi-to-leaf}
The meeting time of the broom graph is 
$$
\Tmeet(\Bnd) 
=
(2 d-2) n  + \frac{2d^3-6d^2+4d}{3(n-1)} -2 d^2+2 d+\frac{1}{2}.
$$
\end{cor}

\begin{proof}
Divide equation \eqref{eqn:join-broom-tip} by $2|E|=2(n-1)$ and simplify.
\end{proof}

We make some observations about meeting times for brooms.  We have $\broom{n}{n-1} = P_n$, and in this case the formula simplifies to 
$$
\Tmeet(P_n) = \frac{4n^2 - 8n +3}{6}.
$$
Meanwhile, we have $\broom{n}{2} = S_n$, and the formula becomes
$$
\Tmeet(S_n) = 2n - \frac{7}{2}.
$$
These two values match the formulas of Theorem \ref{thm:meet-path-star}. 
More generally, it is instructive to consider a sequence of brooms $\{ \broom{n}{d(n)} \}$ for increasing values of $n$, where $d=d(n)$ is a function of $n$. Using asymptotic notation, we have
$$
\Tmeet(\broom{n}{d(n)}) = 
\begin{cases}
    (2d-2) n + o(n) & \mbox{when } d \mbox{ is constant}, \\
     2n d(n) + o(n \, d(n)) & \mbox{when } \omega(1) =  d(n) = o(n) \mbox{ is sublinear}, \\
     \frac{6c-6c^2+2c^3}{3} n^2  + O(n) & \mbox{when } d(n) = cn + o(n) \mbox{ where } 0 < c \leq 1. \\
\end{cases}
$$

\subsection{Proof of Theorem \ref{thm:max-meet}}

We prove two lemmas and then prove Theorem \ref{thm:max-meet}.
It will be fruitful to consider a particular family of rooted trees en route to proving the theorem about trees with diameter $d$. 

\begin{definition}
Let $\Anr$ be the set of rooted trees $(G,z)$, where $z \in V(G)$, on $n$ vertices such that 
$r = \max_{v \in V} d(v,z)$. 
\end{definition}

\begin{lemma}
\label{lemma:pi-to-leaf-r}
Suppose that $(G,z) \in \Anr$ achieves $\max_{(G,z) \in \Anr}  \J(z)$. Let $G_1, G_2, \ldots, G_k$ be the components of $G-z$.
Then for $1 \leq i \leq d$, we have
$\max_{v \in V(G_i)} d(v, z)=r$. 
\end{lemma}

\begin{proof}
Without loss of generality, $\max_{v \in V(G_1)} d(v, z)=r$. Let
$v_0, v_1, \ldots, v_r=z$ be a path of length $r$ where $v_0, v_1, \ldots, v_{r-1} \in V(G_1)$. Assume for the sake of contradiction that there exists $2 \leq i \leq k$ such that
$\max_{v \in V(G_i)} d(v, z)=s<r$. Let $y \in V(G_i)$ be the unique neighbor of $z$, and set $\newG$ to be the graph with $V(\newG)=V(G)$ and $E(\newG) = E(G) - (z,y) + (v_s,y)$.

Observe that $(\newG,z) \in \Anr$. Furthermore, it is clear that $\newH(v,z) \geq H(v,z)$ for all $v \in V(G)$. 
Because $\newH(z,z) = 0 = H(z,z)$, we have
\begin{align*}
\newJ( z) 
&= 
\sum_{v \in V}  \newdeg(v) \newH(v,z) \\
&= 
\sum_{v \in V \backslash \{z,v_s\}}  \deg(v) \newH(v,z) 
+ (\deg(v_s)+1) \newH(v_s,z) \\
&>
\sum_{v \in V \backslash \{z,v_s\}}  \deg(v) H(v,z) 
+ \deg(v_s) H(v_s,z) = \J(z),
\end{align*}
contradicting the maximality of $(G,z)$.
\end{proof}

\begin{lemma}\label{lemma:broomify}
The quantity
$
\max_{(G,z) \in \Anr}  \J(z)
$ 
is achieved uniquely by the broom graph $B_{n, r}$ where $z=v_r$ is the final leaf of the broom handle. 
\end{lemma}

\begin{proof}
 We prove the result using double induction over the order $n$ and $r = \max_{v \in V} d(v,z)$.

For our base cases, consider $r=2$ and any value of $n\geq 3$. The unique tree in $\mathcal{A}_{n, 2}$ is the star $S_n = B_{n,2}$, and the statement holds.

Let $n' < n$ and $r' \leq r$. Assume that among the rooted trees in $\mathcal{A}_{n', r'}$, the maximum value of $H(\pi, z)$ is achieved by $B_{n', r'}$ rooted at its final handle vertex $z'$.

We now prove the inductive step.
Suppose that $(G,z) \in \Anr$ achieves $\max_{(G,z) \in \Anr} \J(z)$. Let $G_1, \ldots, G_k$ be the components of $G-z$. For $1 \leq i \leq k$, let $n_i = |V(G_i)|$, where
$n_1 \geq n_2 \geq \cdots \geq n_k$. Finally, let $y_i \in V(G_i)$ be the unique neighbor of $z$ in $V(G_i)$. 

By Lemma \ref{lemma:pi-to-leaf-r}, for $1 \leq i \leq k$, we have $\max_{v \in V(G_i)} d(v, z)=r$, and hence
$\max_{v \in V(G_i)} d(v, y_i)=r-1$. Therefore, $(G_i,y_i) \in \mathcal{A}_{n_i,r-1}$. Next, we observe that
\begin{align}
    \nonumber
    \J(z) 
    &= 
    \sum_{v \in V} \deg(v) H(v,z) 
    = \sum_{i=1}^k \sum_{v \in V(G_i)} \deg(v) (H(v,y_i) + H(y_i, z)) \\
    \nonumber
    &= 
    \sum_{i=1}^k \big( \J_{G_i} (y_i) + \deg(V(G_i)) \, H(y_i,z)  \big) \\
    \label{eqn:join-partition}
    &=
    \sum_{i=1}^k \left( \J_{G_i} (y_i) + (2n_i-1)^2  \right)
\end{align}
where the last equality follows from equation \eqref{eqn:hit-adjacent} for the hitting time between adjacent vertices.

Since $J(z)$ is the maximum value, by the inductive hypothesis, each $(G_i,y_i) \in \mathcal{A}_{n_i,r-1}$ must achieve $\max_{\mathcal{A}_{n_i,r-1}} \J(z)$, and therefore each $G_i$ is the broom graph $B_{n_i,r-1}$. If $k=1$, then $G$ is the broom $B_{n,r}$, and we are done. So assume, for the sake of contradiction, that $k>1$. We will show that moving a bristle leaf from broom $G_2$ to broom $G_1$ will increase $J(z)$, and this contradiction will complete the proof. There are two cases, depending on whether $G_2$ has multiple bristles.

{\bf Case 1:} $G_1 = B_{n_1,r-1}$ and $G_2=B_{n_2,r-1}$ where $n_1 \geq n_2 \geq r+1$ (so that $G_2$ has at least two bristles). We move a bristle leaf from $G_2$ to $G_1$. Define $\newG$ to be the graph obtained by setting $\newG_1 = B_{n_1+1,r-1}$ and $\newG_2=B_{n_2-1,r-1}$, and $\newG_i = G_i$ otherwise. 
By equation \eqref{eqn:join-partition},we have
\begin{equation}
\label{eqn:max-meet-diff}
 \newJ(z) - J(z) = 
\begin{multlined}[t]
\newJ_{G_1}(y_1) + (2(n_1+1)-1)^2  + \newJ_{G_2}(y_2) + (2(n_2-1)-1)^2 \\
-\J_{G_1}(y_1) - (2n_1-1)^2  - \J_{G_2}(y_2) - (2n_2-1)^2. 
\end{multlined}   
\end{equation}
By equation \eqref{eqn:join-broom-tip} with $d=r-1$, we have
\begin{equation}
\label{eqn:pi-to-leaf1}
\newJ_{G_1}(y_1) - \J_{G_1}(y_1)
= 8n_1(r-2) +12r-4r^2 -7  
\end{equation}
and
\begin{equation}
\label{eqn:pi-to-leaf2}
\newJ_{G_2}(y_2) - \J_{G_2}(y_2)
= -8n_2(r-2) -4r+4r^2 -9.
\end{equation}
We also have
\begin{equation}
\label{eqn:pi-to-leaf3}
(2(n_1+1)-1)^2 - (2n_1-1)^2 = 8n_1
\end{equation}
and
\begin{equation}
\label{eqn:pi-to-leaf4}
(2(n_2-1)-1)^2 - (2n_2-1)^2 = 8-8n_2.
\end{equation}
Adding these four expressions yields
$$
\newJ(z) - \J(z)
= 8(n_1-n_2+1) (r-1) > 0,
$$
contradicting the maximality of $G$.

{\bf Case 2:} $G_1=B_{n_1,r-1}$ and $G_2=B_{r,r-1} = P_r$, where $n_1 \geq r$. We move the endpoint of the path $G_2$ to become a bristle leaf in $G_1$. Equation \eqref{eqn:max-meet-diff} still applies. However,
equation \eqref{eqn:pi-to-leaf2} is now calculated using the meeting time formula in Theorem \ref{thm:meet-path-star}: for a path $P_s$ on vertices $v_0, v_1, \ldots, v_{s-1}$, we have
$$
J_{P_s} (v_{s-1}) = 2(s-1) \frac{4s^2-8s+3}{6} = \frac{4s^3 - 12s^2 + 11s -3}{3}.
$$
We have $G_2=P_r$ and $\newG_2=P_{r-1}$
Therefore
\begin{align*}
\newJ_{G_2}(y_2) - \J_{G_2}(y_2)  
&=
\J_{P_{r-1}}(v_{r-2}) - \J_{P_r}(v_{r-1}) = -(2r-3)^2
\end{align*}
Adding this difference to expressions \eqref{eqn:pi-to-leaf1}, \eqref{eqn:pi-to-leaf3}, and \eqref{eqn:pi-to-leaf4} with $n_2=r$ gives
$$
\newJ(z) - \J(z)
= 8 (n - r+1) (r - 1)  > 0,
$$
once again contradicting the maximality of $G$. This completes the proof.
\end{proof}

We can now prove Theorem \ref{thm:max-meet}.

\begin{proof}
We prove the statement for $\Jmax(G) = 2|E| \cdot  \Tmeet(G)$.
Given $G \in \Tnd$, let $x$ be the vertex achieving $\Jmax(G) = \max_{v \in V} \J(z)$. This vertex $x$ is a leaf by Corollary \ref{cor:leaf-is-max}. Define $r = \max_{v \in V} d(v,x) \leq d$ and root the graph at $x$, so that $G \in \Anr$. By Lemma \ref{lemma:broomify},  the unique maximizer of the joining time for trees in $\Anr$ is $(\Bnr,z)$ where $z$ is the handle tip vertex $v_r$. If $r < d$, then we can replace $\broom{n}{r}$ by $\broom{n}{d}$ and increase $\J(z)$ by Corollary \ref{cor:max-join-increasing}. This proves the first statement of the theorem. The formula for $\Tmeet(\Bnd)$ was proven in Corollary \ref{cor:broom-pi-to-leaf}.
\end{proof}


\section{Minimizing the meeting time}

In this section, we prove Theorems \ref{thm:min-meet-diam-3} and  \ref{thm:min-meet}.
To start, we develop formulas for $\Tmeet(\Dnd)$ and $\Tmeet(\Dnd')$, and quickly prove Theorem \ref{thm:min-meet-diam-3}, which characterizes $\min_{G \in \tree{n}{3}} \Tmeet(G)$.  Then we turn to $\min_{G \in \tree{n}{d}} \Tmeet(G)$ for $d \geq 4$. First, we show that a caterpillar achieves this value. Next, we move leaves in pairs to narrow our range of optimal trees to double brooms and near double brooms. Finally, we prove Theorem \ref{thm:min-meet}.

\subsection{The meeting time for double brooms and near double brooms}

In this section, we develop formulas for $\Tmeet(\Dnd)$ 
and $\Tmeet(\Dnd')$, 
and  prove Theorem \ref{thm:min-meet-diam-3} about the meeting time of double stars. 
We start by setting some notation.
Recall from equation \eqref{eqn:adjacent-set} that when $u,v$ are adjacent vertices, we use $V_{u:v}$ and $V_{v:u}$ to denote the vertices in the components of $G- (u,v)$.
We introduce some helpful notation for decomposing a caterpillar along its spine, as shown in Figure \ref{fig:left-right}.

\begin{figure}
\begin{center}
    \begin{tikzpicture}

\begin{scope}

\begin{scope}[shift={(1,0)}]
\foreach \x in {140,165,195,220}
{
\draw[thick] (0,0) -- (\x:1);
\draw[fill] (\x:1) circle (2pt);
}
\end{scope}

\begin{scope}[shift={(2,0)}]
\foreach \x in {-75,-105}
{
\draw[thick] (0,0) -- (\x:1);
\draw[fill] (\x:1) circle (2pt);
}
\end{scope}

\begin{scope}[shift={(3,0)}]
\foreach \x in {-90}
{
\draw[thick] (0,0) -- (\x:1);
\draw[fill] (\x:1) circle (2pt);
}
\end{scope}

\begin{scope}[shift={(4,0)}]
\foreach \x in {-70,-90,-110}
{
\draw[thick] (0,0) -- (\x:1);
\draw[fill] (\x:1) circle (2pt);
}
\end{scope}

\begin{scope}[shift={(6,0)}]
\foreach \x in {-30,0,30}
{
\draw[thick] (0,0) -- (\x:1);
\draw[fill] (\x:1) circle (2pt);
}
\end{scope}

\draw[thick] (1,0) -- (6,0);

\foreach \x in {1,2,3,4,5,6}
{
\draw[fill] (\x,0) circle (2pt);
\node[above] at (\x,0) {\scriptsize $v_{\x}$};
}

\node[above] at (0,.25) {\scriptsize $v_0$};
\node[above] at (7,0) {\scriptsize $v_7$};

\draw[rounded corners, dashed] (-0.25,-1) rectangle (1.25, 1);
\draw[rounded corners, densely dotted] (-0.5,-1.35) rectangle (4.65, 1.35);
\draw[rounded corners, dashed] (3.35,-1.2) rectangle (7.5, 1);
\draw[rounded corners, densely dotted] (5.75,-.75) rectangle (7.25, .75);

\node at (0.75,-0.75) {\scriptsize $L_1$};
\node at (2.5,0.85) {\scriptsize $L_4$};
\node at (5.25,-0.8) {\scriptsize $R_4$};
\node at (6.25,-0.5) {\scriptsize $R_6$};

\end{scope}

\end{tikzpicture}
\end{center}
    \caption{A caterpillar with  example left sets $L_1$, $L_4$ and right sets $R_4$, $R_6$.}
    \label{fig:left-right}
\end{figure}
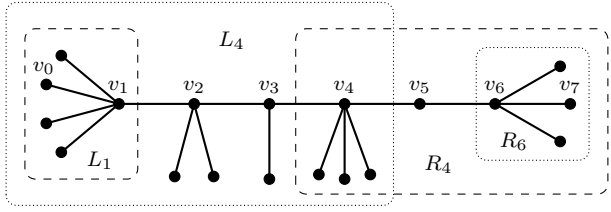

\begin{definition}
Let $G$ be a caterpillar $G \in \Tnd$ with spine $P = \{ v_0, v_1, \ldots, v_d \}$. For
$0 \leq i \leq d-1$, we define the $i$th \emph{left set}
$
L_i = V_{v_i:v_{i+1}}
$
and set $\ell = |L_1|-1$ to be the number of leftmost leaves. For $1 \leq j \leq d$ we define the $j$th \emph{right set}
$
R_j = V_{v_j:v_{j-1}}
$
and set $r = |R_{d-1}|-1$ to be the number of rightmost leaves.
\end{definition}

We now turn our attention to double brooms. Our next lemma provides joining time formulas for the endpoints of a double broom. 

\begin{lemma}
    \label{lemma:double-broom-join-time}
    Let $G\in \Tnd$ be a double-broom with spine $P= \{v_0,\ldots,v_d \}$ where $d>2$. Let $\ell$ and $r$ be the number of left and right leaves, respectively. Then
    \begin{equation}
    \label{eqn:hitting-time-right-side}
        \J(v_0)=4 n^2-11 n-d +9 + \sum_{i=1}^{d-2}(2r+2i-1)^2.
    \end{equation}
    and
    \begin{equation}
        \label{eqn:hitting-time-left-side}
        \J(v_d)=4 n^2-11 n-d +9 + \sum_{i=1}^{d-2}(2\ell+2i-1)^2.
    \end{equation}
\end{lemma}

\begin{proof}
Observe that 
\begin{equation}
\label{eqn:join-double-broom-1}
     J(v_1)=J(L_1,v_1)+J(R_1,v_1) = \ell +  J(R_1,v_1).   
\end{equation}
We expand
$$
\J(R_1,v_1)=\J(R_2,v_2)+\deg(R_2)H(v_2,v_1)=\J(R_2,v_2)+\deg(R_2)^2
$$
where the last equality follows from equation \eqref{eqn:hit-adjacent}.
Repeating this process for $2 \leq j \leq d-1$ yields 
$$
\J(R_1,v_1)=\J(R_{d-1},v_{d-1})+\sum_{j=2}^{d-1}\deg(R_j)^2 = r + \sum_{j=2}^{d-1}\deg(R_j)^2.
$$
Next, we observe that  for $2\leq j\leq d-1$
$$
\deg(R_j)=2(d-j)+ (2r -1)
$$
and therefore 
$$
J(v_1)=\ell+r+\sum_{j=2}^{d-1}(2r+2(d-j))-1)^2
=n-d+1+\sum_{i=1}^{d-2}(2r+2i-1)^2.
$$    
    Applying Lemma \ref{lemma:unif-to-leaf} gives the formula
    \begin{align*}
        J(v_0)
        & 
        = J(v_1) + 4(n^2-3n+2) 
         = 4 n^2-11 n-d +9 + \sum_{i=1}^{d-2}(2r+2i-1)^2.   
    \end{align*}

    Equivalent logic shows that formula \eqref{eqn:hitting-time-left-side} holds.
\end{proof}

Recall that the path $P_n$ is the balanced double broom $D_{n,n-1}$ with diameter $d=n-1$ and with $\ell=r=1$. In this case, formula \eqref{eqn:hitting-time-right-side} simplifies to
$$
\Jmax(P_n) = \frac{4 n^3}{3}-4 n^2+\frac{11 n}{3}-1
$$
and dividing by $2|E|=2(n-1)$ yields
$$
\Tmeet(P_n) = \frac{1}{6} \left(4 n^2-8 n+3\right),
$$
which matches the path meeting time of Theorem \ref{thm:meet-path-star}. 

Lemma \ref{lemma:double-broom-join-time} is sufficient to prove Theorem \ref{thm:min-meet-diam-3}, which characterizes $\min_{G \in \tree{n}{3}} \Tmeet(G)$ for the family of double stars.

\begin{proof}[Proof of Theorem \ref{thm:min-meet-diam-3}]
For a double star with diameter $d=3$, formulas \eqref{eqn:hitting-time-right-side} and \eqref{eqn:hitting-time-left-side} simplify to
\begin{align*}
    J(v_0) &=  4 n^2- 11 n + 6 +  (2r+1)^2, \\
    J(v_3) &= 4 n^2 - 11 n + 6 + (2\ell+1)^2.  
\end{align*}
Without loss of generality, we have $\ell \leq r$, so $J(v_0) \geq J(v_3)$. Furthermore, the double star that minimizes $J(v_0)$ is the balanced double star with $\lfloor (n-2)/2 \rfloor$ left leaves and $\lceil (n-2)/2 \rceil$ right leaves. We have
$$
\Jmax(\double{n}{3}) = J(v_0) =  4 n^2 - 11 n + 6 + \left(2 \left\lceil \frac{n-2}{2} \right\rceil +1 \right)^2.
$$
When $n$ is even, this simplifies to
$$
\Jmax(\double{n}{3}) =J(v_0) = 4 n^2- 11 n + 6  + (n-1)^2 =   5 n^2 - 13 n + 7 
$$
with corresponding meeting time
$$
\Tmeet(\double{n}{3}) = H(\pi, v_0) = \frac{5 n^2 - 13 n + 7}{2(n-1)} = \frac{5}{2}n - 4 - \frac{1}{2(n-1)}.
$$
When $n$ is odd, this simplifies to
$$
\Jmax(\double{n}{3}) = J(v_0) = 4 n^2- 11 n + 6  + n^2 =   5 n^2 - 11 n + 6 
$$
with corresponding meeting time
$$
\Tmeet(\double{n}{3}) = H(\pi, v_0) = \frac{5 n^2 - 11 n + 6 }{2(n-1)} = \frac{5}{2}n - 3.
$$
This completes the proof for double stars.
\end{proof}

From this point forward, we assume that $d \geq 4$. We calculate the meeting time for the balanced double broom $\Dnd$ when $n$ and $d$ have opposite parity.

\begin{lemma}
\label{lemma:double-broom-meet-time}
Suppose that $4 \leq d < n$ where $n$ and $d$ have opposite parity. The meeting time of $\Dnd$ is
$$
\Tmeet(\Dnd) = \frac{1}{2}(d+2) n
+ \frac{1}{6(n-1)} \left( d^3-6 d^2+8 d \right) - \frac{1}{2} (d+5).
$$
\end{lemma}

\begin{proof}
 When $n$ and $d$ have opposite parity, the balanced double broom $\Dnd$ has $\ell=r=(n-d+1)/2$. In this case, joining time equation \eqref{eqn:hitting-time-right-side} simplifies to 
$$
J(v_0) = (d+2) n^2-(2d+7)n + \frac{1}{3}( d^3 -6 d^2+ 11d) +5.
$$
We have $\Tmeet(\Dnd) = J(v_0)/2(n-1)$ and the formula above follows via simplification.
\end{proof}

Now, we direct our attention to near double brooms. We start with joining time formulas for the endpoints of a near double broom.

\begin{lemma}
    \label{lemma:near-double-broom-join-time}
    Let $G\in \Tnd$ be a near double broom with spine $P=\{ v_0,\ldots,v_d \}$ where $d \geq 4$. Let $\ell$ and $r$ be the number of left and right leaves, respectively, and let the singleton leaf be adjacent to $v_k$ where $2 \leq k \leq d-2$. Then
    \begin{equation}
    \label{eqn:near-hitting-time-right-side}
        \J(v_0)=4 n^2-11 n-d +9 + \sum_{i=1}^{d-k-1}(2r+2i-1)^2 + \sum_{i=d-k}^{d-2}(2r+2i+1)^2.  
    \end{equation}
    and
    \begin{equation}
        \label{eqn:near-hitting-time-left-side}
        \J(v_d)=4 n^2-11 n-d +9 + \sum_{i=1}^{k-1}(2\ell+2i-1)^2 + \sum_{i=k}^{d-2}(2\ell+2i+1)^2. 
    \end{equation}
\end{lemma}

\begin{proof}
The proof is similar to that of the Lemma \ref{lemma:double-broom-join-time}. Equation \eqref{eqn:join-double-broom-1} holds, and 
$$
\J(R_1,v_1)=\J(R_{d-1},v_{d-1})+\sum_{j=2}^{d-1}\deg(R_j)^2 + H(z,v_k) = r + 1 + \sum_{j=2}^{d-1}\deg(R_j)^2.
$$
We have
$$
\deg(R_j)=
\begin{cases}
   2(d-j)+2r+1 & \mbox{for } 2 \leq j \leq k, \\ 
   2(d-j)+2r-1 & \mbox{for } k+1 \leq j \leq d-1. \\  
\end{cases}
$$
Therefore 
$$
  J(v_1)=\ell + r + 1 + \sum_{j=2}^{k}(2(r+d-j)+1)^2 + \sum_{j=k+1}^{d-1}(2(r+d-j)-1)^2 
$$
Reindexing by $i=d-j$ and then using $\ell + r + 1 = n-d+1$ and Lemma \ref{lemma:unif-to-leaf} gives
    \begin{align*}
        J(v_0)
        &= J(v_1) + 4|E|(|E|-1) = J(v_1) + 4(n^2-3n+2) \\
        & = 4 n^2-11 n-d +9 + \sum_{i=d-k}^{d-2}(2r+2i+1)^2 + \sum_{i=1}^{d-k-1}(2r+2i-1)^2.          
    \end{align*}

    Equivalent logic shows that formula \eqref{eqn:near-hitting-time-left-side} holds.
\end{proof}

\begin{lemma}
\label{lemma:near-double-broom-meet-time}
Suppose that $4 \leq d < n$ where $n$ and $d$ have the same parity. The meeting time of the balanced near double broom $\Dnd'$ is
$$
\Tmeet(\Dnd') = 
\frac{1}{2}\left(d+2\right)
+ \frac{d^3-3d^2-d+6}{6(n-1)}
-
\begin{cases}
  \frac{1}{2}(d+5) & \mbox{if $d$ is even}, \\
 \frac{1}{2}(d+3)
 & \mbox{if $d$ is odd}. 
\end{cases}
$$
\end{lemma}

\begin{proof}
 When $n$ and $d$ have the same parity, the balanced near double broom $\Dnd'$ has $\ell=r=(n-d)/2$. 
When $d$ is even, the singleton leaf is adjacent to $v_{d/2}$, and the joining time equation \eqref{eqn:near-hitting-time-left-side} simplifies to 
$$
J(v_d) = (d+2) n^2 -(2d+7 )n + \frac{1}{3}( d^3-3d^2 + d)+7.
$$ 
On the other hand, when $d$ is odd, the singleton leaf is adjacent to $v_{(d-1)/2}$ and the joining time equation \eqref{eqn:near-hitting-time-left-side} simplifies to 
$$
J(v_d) = (d+2)n^2 - (2d+5)n + \frac{1}{3}( d^3-3d^2 + 2d)+4.
$$
We have $\Tmeet(\Dnd') = J(v_d)/2(n-1)$ and the formulas above follow via simplification.
\end{proof}

We conclude this subsection by considering the asymptotic behavior of $\Tmeet(D_{n,d(n)})$ and $\Tmeet(D_{n,d(n)}')$.
Consider a sequence of balanced brooms $\{ \broom{n}{d(n)} \}$ for increasing values of $n$, where $d=d(n)$ is a function of $n$. Using asymptotic notation, we have
$$
\Tmeet(D_{n,d(n)}) =
\begin{cases}
    \frac{d+2}{2} n + o(n) & \mbox{when } d \mbox{ is constant,} \\
    \frac{1}{2} n d(n) + o(n \, d(n)) & \mbox{when } \omega(1) = d(n) = o(n) \mbox{ is sublinear,} \\
    \frac{3c-c^3}{6} n^2 + o(n^2) & \mbox{when } d(n) = cn + o(n) \mbox{ where } 0 < c \leq 1.
\end{cases}
$$
The same asymptotic formulas hold for $\Tmeet(D_{n,d(n)}')$.
These asymptotic results are very similar to the analogous values for the mixing time of balanced double broom, see \cite{BHOV}.

\subsection{From tree to caterpillar}

In this section, we show that for any non-caterpillar $G \in \Tnd$ where $d \geq 4$, there exists a caterpillar $\newG \in \Tnd$ with a smaller joining time $\Jmax(\newG) < \Jmax(G)$. 
We start by showing that the joining time of a caterpillar is achieved by one of the two spine leaves.

\begin{lemma}
    \label{lemma:max-on-caterpillar}
    For a caterpillar $G\in \Tnd$ with spine $P=\{ v_0,\ldots,v_d \}$, we have
    $$
    \max_{v \in V} J(v) = \max \{ J(v_0), J(v_d) \}. 
    $$
\end{lemma}

\begin{proof}
By Corollary \ref{cor:leaf-is-max}, $\max_{v \in V} J(v)$ is achieved by a leaf $z$. Meanwhile, Lemma \ref{lemma:unif-to-leaf} shows that
$$
J(z) = J(v_i) + 4(n^2-3n+2)
$$
where $v_i$ is the neighbor of $z$, where $1 \leq i \leq d-1$. So it is sufficient to show that
\begin{equation}
\label{eqn:max-on-cat-proof}
\max_{1 \leq i \leq d-1} J(v_i) = \max \{ J(v_1), J(v_{d-1}) \}.
\end{equation}
For $2 \leq i \leq d-2$,  we either have $\deg(V_{v_{i-1}: v_i}) < \deg(V_{v_i:v_{i-1}})$ or
$\deg(V_{v_{i+1}: v_i}) < \deg(V_{v_i:v_{i+1}})$.
In the former case, Lemma \ref{lemma:adjacent-join-times}  yields $J(v_{i}) < J(v_{i-1})$. We apply this lemma repeatedly to find that $J(v_{i}) < J(v_{i-1}) < \cdots < J(v_1)$. 
In the latter case, we similarly have $J(v_{i}) < J(v_{i+1}) < \cdots < J(v_{d-1})$. 

We conclude that equation \eqref{eqn:max-on-cat-proof} is valid. Every leaf adjacent to the maximizing vertex (including either $v_0$ or $v_d$) achieves $\max_{v \in V} J(v).$
\end{proof}

We update our tree $G \in \Tnd$ by moving one or two leaves at a time. We refer to these leaf operations as \emph{tree surgeries}. We have the following definition and lemma from \cite{BHOV}.

\begin{definition}
Let $G \in \tree{n}{d}$ be a non-caterpillar, with maximal path $P = \{ v_0, v_1, \ldots, v_s \}$. Let $y \in V(G \backslash P)$ be a leaf that is not adjacent to $P$. Define the tree surgery $\sigma(G,P;y)$ to be the tree obtained by moving  leaf $y$ to be adjacent to path vertex $v_k$ that is closest to $y$. 
\end{definition}

An example  of $\sigma(G,P;y)$ is shown in Figure \ref{fig:surgery-sigma}. 

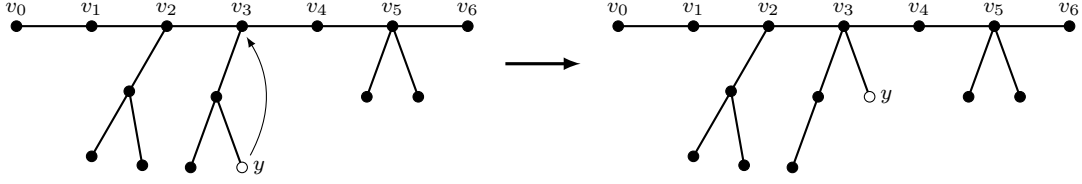
\begin{figure}
\begin{center}
\begin{tikzpicture}

\begin{scope}

\begin{scope}[xscale=-1, shift={(-6,0)}]

\draw[thick] (0,0) -- (6,0);

\foreach \x in {0,1,2,3,4,5,6}
{
\draw[fill] (\x,0) circle (2pt);
\node[above] at (6-\x,0) {\scriptsize $v_{\x}$};
}

\begin{scope}[shift={(3,0)}]

\draw[thick] (0,0) -- (-70:2);
\draw[fill] (-70:1) circle (2pt);
\draw[fill] (-70:2) circle (2pt);

\begin{scope}[shift={(-70:1)}]

\draw[thick] (0,0) -- (-110:1);
\draw[fill=white] (-110:1) circle (2pt);
\node[right] at (-110:1) {\scriptsize $y$};

\draw[-latex] (-120:.9) to [bend left] (-.4,.8);

\end{scope}

\end{scope}

\begin{scope}[shift={(1,0)}]
\draw[thick] (0,0) -- (-70:1);
\draw[thick] (0,0) -- (-110:1);
\draw[fill] (-70:1) circle (2pt);
\draw[fill] (-110:1) circle (2pt);
\end{scope}

\begin{scope}[shift={(4,0)}]
\draw[thick] (0,0) -- (-60:2);
\draw[fill] (-60:1) circle (2pt);
\draw[fill] (-60:2) circle (2pt);

\begin{scope}[shift={(-60:1)}]
    \draw[thick] (0,0) -- (-100:1);
    \draw[fill] (-100:1) circle (2pt);
\end{scope}

\end{scope}

\end{scope}

\draw[very thick, -latex] (6.5, -.5) -- (7.5,-.5);

\begin{scope}[xscale=-1, shift={(-14,0)}]

\draw[thick] (0,0) -- (6,0);

\foreach \x in {0,1,2,3,4,5,6}
{
\draw[fill] (\x,0) circle (2pt);
\node[above] at (6-\x,0) {\scriptsize $v_{\x}$};
}

\begin{scope}[shift={(3,0)}]

\draw[thick] (0,0) -- (-70:2);
\draw[fill] (-70:1) circle (2pt);
\draw[fill] (-70:2) circle (2pt);

\begin{scope}

\draw[thick] (0,0) -- (-110:1);
\draw[fill=white] (-110:1) circle (2pt);
\node[right] at (-110:1) {\scriptsize $y$};

\end{scope}

\end{scope}

\begin{scope}[shift={(1,0)}]
\draw[thick] (0,0) -- (-70:1);
\draw[thick] (0,0) -- (-110:1);
\draw[fill] (-70:1) circle (2pt);
\draw[fill] (-110:1) circle (2pt);
\end{scope}

\begin{scope}[shift={(4,0)}]
\draw[thick] (0,0) -- (-60:2);
\draw[fill] (-60:1) circle (2pt);
\draw[fill] (-60:2) circle (2pt);

\begin{scope}[shift={(-60:1)}]
    \draw[thick] (0,0) -- (-100:1);
    \draw[fill] (-100:1) circle (2pt);
\end{scope}

\end{scope}

\end{scope}












 \end{scope}

\end{tikzpicture}
\end{center}

\caption{An example of tree surgery $\sigma(G,P;y)$}
\label{fig:surgery-sigma}
\end{figure}

\begin{lemma}[Lemma 4.5 of \cite{BHOV}]
\label{lemma:caterpillarify-hitting-times}
Let $G \in \tree{n}{d}$ be a non-caterpillar with maximal path $P = \{ v_0, v_1, \ldots, v_s \}$, where $s \leq d$, and leaf $y \in V(G \backslash P)$ that is not adjacent to $P$. If $\newG = \sigma(G,P;y)$ then the following statements hold.
\begin{enumerate}[(a)]
\item $\newH(v_i, v_j)=H(v_i,v_j)$, for all $0 \leq i,j  \leq s$.
\item More generally, for $0 \leq j \leq s$, we have $\newH(v,v_j) \leq H(v,v_j)$ for all $v \in V(G)$, 
with  $\newH(y,v_j) < H(y,v_j)$, in particular.
\item $\newH(\newpi, v_j) < H(\pi, v_j)$ for $0 \leq j \leq s$. 
\end{enumerate}
\end{lemma}

\begin{lemma}
    \label{lemma:tree-to-caterpillar}
Let $G \in \Tnd$ be a non-caterpillar.  Then there is a caterpillar $\newG \in \Tnd$ such that
$\Jmax(\newG) < \Jmax(G)$.
\end{lemma}
\begin{proof}
Let $P= \{ v_0, v_1, \ldots, v_d \}$ be a geodesic of $G$, and let $y \in V$ be a leaf that is not adjacent to $P$. Set  $\newG = \sigma(G,P;y)$. By Lemma \ref{lemma:caterpillarify-hitting-times}(c),  
$\newH(\newpi, v_j) < H(\pi, v_j)$ for all $1 \leq j \leq s$ 
Equivalently, we have $\newJ( v_j) < J( v_j)$ for all $1 \leq j \leq s$; in particular, $\newJ(v_0) < \J(v_0)$ and $\newJ(v_d) < \J(v_d)$. 

Repeat this process until we have a caterpillar $\newG$ with spine $P$. By Lemma \ref{lemma:max-on-caterpillar},
$$
\Jmax(\newG) = \max \{ \newJ(v_0), \newJ(v_d) \} 
< \max \{ \J(v_0), \J(v_d) \} \leq \max_{v \in V(G)} \J(v) = \Jmax(G).
$$
\end{proof}

\subsection{Proof of Theorem \ref{thm:min-meet}}

In this section, we prove Theorem \ref{thm:min-meet}: for $4 \leq d <n$, the quantity
$\min_{G\in \Tnd}\max_{v\in V} H(\pi,v)$
is uniquely achieved by (a) the balanced double broom $\Dnd$ when $n$ and $d$ have different parity, and (b) the balanced near double broom $\Dnd'$ when $n$ and $d$ have the same parity.

By Lemma \ref{lemma:tree-to-caterpillar}, the optimal tree must be a caterpillar. Hence, for the remainder of this section, $G \in \Tnd$ will be a caterpillar with spine $P=\{v_0, v_1, \ldots, v_d \}.$ We need the following definition and lemma from \cite{BHOV}.

\begin{definition}
Let $G \in \tree{n}{d}$ be a caterpillar with spine $P=\{v_0, v_1, \ldots, v_d\}$. If $G$ has a leaf $x \in V(G \backslash P)$ adjacent to $v_i$, then we define $\tau(G,P;(i,j))$ to be the caterpillar $G - (v_i, x) + (v_j, x)$. If $G$ also has a leaf $y \in V(G \backslash P)$ adjacent to $v_k$, then we define the tree surgery
$\tau(G,P; (i,j) \wedge (k,\ell))$ to be the caterpillar $G - (v_i, x) + (v_j, x) - (v_k,y) + (v_{\ell}, y).$
\end{definition}

An example of surgery $\tau(G,P; (i,i-1) \wedge (j,j-1))$ is shown in Figure \ref{fig:surgery-tau}. 

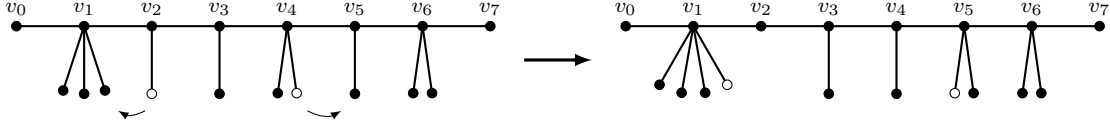
\begin{figure}
\begin{center}
\begin{tikzpicture}[scale=.9]

\begin{scope}

\begin{scope}

\draw[thick] (0,0) -- (7,0);

\foreach \x in {0,1,2,3,4,5,6,7}
{
\draw[fill] (\x,0) circle (2pt);
\node[above] at (\x,0) {\scriptsize $v_{\x}$};
}

\begin{scope}[shift={(1,0)}]

\foreach \x in {-72, -90, -108}
{
\draw[thick] (0,0) -- (\x:1);
}

\draw[fill] (-108:1) circle (2pt);
\draw[fill] (-90:1) circle (2pt);
\draw[fill] (-72:1) circle (2pt);

\end{scope}

\draw[thick] (2,0) -- (2,-1);
\draw[fill=white] (2,-1) circle (2pt);

\draw[-latex] (1.9,-1.25) to [bend left] (1.5,-1.25);

\draw[thick] (3,0) -- (3,-1);
\draw[fill] (3,-1) circle (2pt);

\begin{scope}[shift={(4,0)}]

\foreach \x in {-82,-98}
{
\draw[thick] (0,0) -- (\x:1);
}

\draw[fill] (-98:1) circle (2pt);
\draw[fill=white] (-82:1) circle (2pt);

\end{scope}

\draw[-latex] (4.3,-1.25) to [bend right] (4.8,-1.25);

\draw[thick] (5,0) -- (5,-1);
\draw[fill] (5,-1) circle (2pt);

\end{scope}

\begin{scope}[shift={(6,0)}]

\foreach \x in {-82,-98}
{
\draw[thick] (0,0) -- (\x:1);
}

\draw[fill] (-98:1) circle (2pt);
\draw[fill] (-82:1) circle (2pt);

\end{scope}

\draw[very thick, -latex] (7.5, -.5) -- (8.5,-.5);

\begin{scope}[shift={(9,0)}]

\draw[thick] (0,0) -- (7,0);

\foreach \x in {0,1,2,3,4,5,6,7}
{
\draw[fill] (\x,0) circle (2pt);
\node[above] at (\x,0) {\scriptsize $v_{\x}$};
}

\begin{scope}[shift={(1,0)}]

\foreach \x in {-60,-80,-100,-120}
{
\draw[thick] (0,0) -- (\x:1);
}

\draw[fill] (-120:1) circle (2pt);
\draw[fill] (-100:1) circle (2pt);
\draw[fill] (-80:1) circle (2pt);
\draw[fill=white] (-60:1) circle (2pt);

\end{scope}

\draw[thick] (3,0) -- (3,-1);
\draw[fill] (3,-1) circle (2pt);

\draw[thick] (4,0) -- (4,-1);
\draw[fill] (4,-1) circle (2pt);

\begin{scope}[shift={(5,0)}]

\foreach \x in {-82,-98}
{
\draw[thick] (0,0) -- (\x:1);
}

\draw[fill=white] (-98:1) circle (2pt);
\draw[fill] (-82:1) circle (2pt);

\end{scope}

\begin{scope}[shift={(6,0)}]

\foreach \x in {-82,-98}
{
\draw[thick] (0,0) -- (\x:1);
}

\draw[fill] (-98:1) circle (2pt);
\draw[fill] (-82:1) circle (2pt);

\end{scope}

\end{scope}

\end{scope}

\end{tikzpicture}
\end{center}

\caption{An example of surgery $\tau(G,P; (2,1) \wedge (4,5))$. }
\label{fig:surgery-tau}
\end{figure}

\begin{lemma}[Lemma 4.10 of \cite{BHOV}]
\label{lemma:two-leaf-pi-access-time}
Given a caterpillar $G$ with spine $P=\{v_0, v_1, \ldots, v_d\}$ and leaf $x$ adjacent to $v_i$ and leaf $y$ adjacent to $v_j$, where $2 \leq i \leq j \leq d-2$. Let $\newG = \tau(G,P; (i,i-1) \wedge (j,j+1))$.
Then $\newH(\newpi,v_d) < H(\pi,v_d)$ and $\newH(\newpi,v_0) < H(\pi,v_0)$.
\end{lemma}

We use this lemma to show that the optimizing caterpillar is either a double broom or a near double broom.

\begin{lemma}
\label{lemma:caterpillar-to-double-broom}
Given a caterpillar $G \in \Tnd$ with at least two leaves adjacent to spine vertices $\{v_2, v_3, \ldots, v_{d-2}\}$. Then there is a double broom or a near double broom $\newG \in \Tnd$ such that $\newJ(v_d) < J(v_d)$ and $\newJ(v_0) < J(v_0)$.
\end{lemma}

\begin{proof}
By Lemma \ref{lemma:two-leaf-pi-access-time}, we can perform tree surgery $\tau(G,P; (i,i-1) \wedge (j,j+1)$ where $2 \leq i \leq j \leq d-2$ to obtain $\newG \in \Tnd$ with $\newH(\newpi,v_d) < H(\pi, v_d)$ and $\newH(\newpi,v_0) < H(\pi, v_0)$. We  repeat this process as long as there are at least two leaves adjacent to $\{v_2, v_3, \ldots, v_{d-2}\}$, increasing the meeting time with each surgery. When we are done, we either have a double broom or a near double broom.
\end{proof}

We now optimize over double-brooms and near double-brooms.  
We start with the near double-broom case. We will show that if this near double-broom does not have $\ell = r$ and its singleton leaf is adjacent to $v_{\lfloor d/2 \rfloor}$, then the tree is not optimal. First, we need a lemma about moving a leaf to be adjacent to a target vertex.

\begin{lemma}
\label{lemma:move-leaf}
Let $G=(V,E)$ be a tree.
Let $x,y,z  \in V$ where $z$ is a leaf adjacent to $y$.
Let $G^* = G - (y, z) + (x, z)$.
Then the following hold.
\begin{enumerate}[(a)]
    \item $H^*(v, x) \leq H(v, x)$ for every vertex $v \in V$.
    \item $\newJ(x) < \J(x)$
\end{enumerate}
\end{lemma}

\begin{proof}
(a) 
For a vertex $v \in V \backslash \{x,z\}$, we have
$\ell^*(v, z; x)=0$ and $\ell^*(v, x; x)=0$. Therefore
equation \eqref{eqn:hit-time} gives
\begin{align*}
    H^*(v, x) 
    &= 
    \sum_{w \in V} \ell^*(v, w; x) \deg^*(w) 
    = 
    \sum_{w \in V \backslash \{x,y,z\}} \!\!\!\!\! \ell^*(v,w;x) \deg^*(w)
    + \ell^*(v,y;x) \deg^*(y) \\
    &= 
    \sum_{w \in V \backslash \{x,y,z\}} \!\!\!\!\! \ell(v,w;x) \deg(w)
    + \ell(v,y;x) (\deg(y) -1) \\
    & \leq 
    \sum_{w \in V} \ell(v, w; x) \deg(w) 
    = H(v,x).
\end{align*}
The conclusion holds for both $v = x$ and $v = z$, because $H^*(x, x) = 0 = H(x, x)$ and
$H^*(z,x) = 1 < H(z, x)$.

(b) Observe that  $H^*(x,x) = 0 = H(x,x)$, and consequently
\begin{align*}
    \newJ(x) &=
    \sum_{w \in V \backslash\{x,y\}} \newdeg(w) \newH(w,x)
    + \newdeg(y)  \newH(y,x) \\
    & < 
    \sum_{w \in V \backslash\{x,y\}} \deg(w)  \newH(w,x)
    +  \deg(y) \newH(y,x) \\
    &\leq 
    \sum_{w \in V \backslash\{x,y\}} \deg(w) H(w,x)
    + \deg(y) H(y,x) = J(x)
\end{align*}
where the last inequality follows from part (a).
\end{proof}

Now we prove two lemmas about the joining time for near double brooms.

\begin{lemma}
    \label{lemma:near-unbalanced-leaves}
    If $G$ is a near double-broom with $\ell < r$ then there exists a double broom $\newG$ with $\Jmax(\newG) < \Jmax(G)$.
\end{lemma}

\begin{proof}
    Let $G\in \Tnd$ be a near double-broom with $\ell < r$ and  with singleton leaf $z$
     adjacent to vertex $v_k$ where $2 \leq k \leq d-2$. 
Set $G^* = G-(v_k,z)+(v_1,z)$ to be the double broom attained by moving $z$ to be adjacent to $v_1$.     
 By Lemma \ref{lemma:unif-to-leaf} and Lemma \ref{lemma:move-leaf}(b), 
 $$
 \newJ(v_0) =  \newJ(v_1) + 4(n^2-3n+2) <
  J(v_1) + 4(n^2-3n+2) = J(v_0).
 $$
The tree $\newG$ is a double broom with $\ell+1$ left leaves and $r$ right leaves, where $\ell+1 \leq r$. It is clear from Lemma \ref{lemma:near-double-broom-join-time} that $\newJ(v_0) \geq \newJ(v_d)$.
Finally, we have
$$
\Jmax(\newG) = \newJ(v_0) < J(v_0) \leq \Jmax(G).
$$
Therefore $G$ does not achieve $\min_{G\in \Tnd} \max_{v\in V}J(v)$.
\end{proof}

\begin{lemma}
    \label{lemma:near-unbalanced-singleton}
    Among near double-brooms with $\ell = r$, the balanced near double broom $\Dnd'$ with singleton leaf $z$ adjacent to $v_{\lfloor d/2 \rfloor}$ is the unique minimizer of the joining time.
\end{lemma}

\begin{proof}
Let $G$ be a near double broom with $\ell = r$.
We must show that the optimal neighbor $v_k$ for the singleton leaf $z$ is $k=\lfloor d/2 \rfloor$.
Without loss of generality, we have $2 \leq k \leq \lfloor d/2 \rfloor$. It is clear from Lemma \ref{lemma:near-double-broom-join-time} and the slight asymmetry of $G$ that (a) $J(v_d) > J(v_0)$ and (b) if $k < \lfloor d/2 \rfloor$ then moving $z$ from $v_i$ to $v_{\lfloor d/2 \rfloor}$ decreases $J(v_d)$, while preserving $J(v_d) > J(v_0)$. So the unique best choice is $k=\lfloor d/2 \rfloor$.   
\end{proof}

Next, we consider the joining time for double brooms.  

\begin{lemma}
    \label{lemma:balanced-double-broom-is-best}
    Let $G\in \Tnd$ be an unbalanced double broom. Then the balanced double broom $\newG=\Dnd$ satisfies $\Jmax(\newG) < \Jmax(G)$. Furthermore, if $n$ and $d$ have the same parity, then the balanced near double broom $\newnewG = \Dnd'$ satisfies $\Jmax(\newnewG) < \Jmax(\newG)$.
\end{lemma}
\begin{proof}
Let $G \in \Tnd$ be a double broom with $\ell \leq r$. 
By Lemma \ref{lemma:max-on-caterpillar}, $J(G) = \max \{ J(v_0), J(v_d) \}$. By Lemma \ref{lemma:double-broom-join-time}, minimizing $r-\ell$ achieves the smallest value of $J(G)$ among double brooms in $\Tnd$. So if $t=r+\ell = n-d+1$ is even, then our unique best choice is the balanced double broom $\Dnd$ with $t/2$ leaves at each end. Of course, $n-d+1$ is even if and only if $n$ and $d$ have different parity.

When $t=\ell+r$ is odd (so $n$ and $d$ have the same parity), then the balanced double broom $\Dnd$ with $\ell = (t-1)/2$ left leaves and $r=(t+1)/2$ right leaves achieves the smallest meeting time. However, we claim that we can achieve an even smaller meeting time by moving one leaf $z$ from $v_{d-1}$ to $v_{\lfloor d/2 \rfloor}$, resulting in the balanced near double broom $\Dnd'$. Indeed, we still have $\newnewJ(v_d) \geq \newnewJ(v_0)$, but clearly $\newnewJ(v_d) < \newJ(v_d)$. By Lemma \ref{lemma:near-unbalanced-singleton}, $\Dnd'$ is also the best choice among near double brooms. In conclusion, balanced near double broom $\Dnd'$ (with $(t-1)/2$ left leaves and $(t-1)/2$ right leaves and singleton leaf adjacent to $v_{\lfloor d/2 \rfloor}$) is the optimal choice. 
\end{proof}

We conclude by offering the proof of Theorem \ref{thm:min-meet}.

\begin{proof}[Proof of Theorem \ref{thm:min-meet}.]
We must determine the tree $G \in \Tnd$ that achieves $\min_{G \in \Tnd} \Tmeet(G) = \min_{G \in \Tnd} \max_{v \in V} H(\pi, v)$. Equivalently,
we can find the tree that achieves $\min_{G \in \Tnd} \Jmax(G) = \min_{G \in \Tnd} \max_{v \in V} J(v)$.
By Lemma \ref{lemma:tree-to-caterpillar}, $G$ must be a caterpillar. By Lemma \ref{lemma:caterpillar-to-double-broom}, this caterpillar must be a double broom or a near double broom. 

Suppose that $G$ is a near double broom. If $\ell < r$, then Lemma \ref{lemma:near-unbalanced-leaves} shows that there is either a double broom $\newG$ with $\Jmax(\newG) < \Jmax(G)$. If $\ell=r$, which occurs if and only if $n$ and $d$ have the same parity, then the optimal neighbor for the singleton leaf is $v_{\lfloor d/2 \rfloor}$ by Lemma \ref{lemma:near-unbalanced-singleton}. So the optimal choice is the balanced near double broom $\Dnd'$.

Now suppose that $G$ is a double broom. If $G$ is not balanced, then Lemma \ref{lemma:balanced-double-broom-is-best} shows that the balanced double broom $\Dnd$ satisfies $\Jmax(\Dnd) < \Jmax(G)$. If $n$ and $d$ have different parity, so that $\ell=r$, then we have minimized the joining time. Otherwise, the balanced near double broom $\Dnd'$ satisfies $\Jmax(\Dnd') < \Jmax(\Dnd) < \Jmax(G)$, and it minimizes the joining time.

Finally, the formulas for $\Tmeet(\Dnd)$ and $\Tmeet(\Dnd')$ were established in Lemma \ref{lemma:double-broom-meet-time} and Lemma \ref{lemma:near-double-broom-meet-time}, respectively. 
\end{proof}

\section{The  meeting time for trees of fixed order}

In this brief section contains the proof of Theorem \ref{thm:max-trees}. We consolidate our results to give the upper and lower bounds for $\Tmeet(G)$ when $G \in \Tn$ is a tree on $n$ vertices.

\begin{proof}[Proof of Theorem \ref{thm:max-trees}]
Let $G \in \Tn$. We consider the lower bound. 
The meeting time for the star is $\Tmeet(S_n)=2n-7/2$ by Theorem \ref{thm:meet-path-star}. This value is smaller than the values for $\Tmeet(D_{n,3})$ provided in Theorem \ref{thm:min-meet-diam-3}.

Now consider the Theorem \ref{thm:min-meet} formulas for $\Tmeet(\Dnd)$ for $d \geq 4$. We consider formula \eqref{eqn:min-meet-opposite-parity} for $n$ and $d$ with opposite parities.

First, we suppose that $d$ is even and $n$ is odd. Plugging $d=2$ into this formula results in 
the value for $\Tmeet(S_n)$. More generally, observe that
\begin{equation}
\label{eqn:max-trees-by-2}
  \Tmeet(D_{n,d+2}) - \Tmeet(D_{n,d}) =  (n-1)^2 + \frac{d(d-2)}{n-1} > 0  
\end{equation}
so this function is increasing for even $d$. Hence it is strictly larger than $\Tmeet(S_n)$ for even $d \geq 4$. 

Second, we suppose that $d$ is odd and $n$ is even. Plugging $d=3$ in formula \eqref{eqn:min-meet-opposite-parity} results in the value for $\Tmeet(D_{n,3})$ for even $n$ in Theorem \ref{thm:min-meet-diam-3}, which we already noted is larger than $\Tmeet(S_n)$. Equation \eqref{eqn:max-trees-by-2} shows that the function is increasing for odd $d$. Hence, the value is strictly larger than $\Tmeet(S_n)$ for odd $d \geq 5$.

The argument for $n$ and $d$ with the same parities, using equation \eqref{eqn:min-meet-same-parity}, is analogous. We conclude that the star $S_n$ is the unique minimizer of the meeting time.

We now turn to the upper bound. By Theorem \ref{thm:max-meet}, we can restrict our attention to double brooms.  By Corollary \ref{cor:max-join-increasing}, the maximum joining time $\Jmax(B_{n,d})$ is strictly increasing for $2 \leq d \leq n-1$. Therefore the unique maximizer is $P_n = B_{n,n-1}$. Of course, it follows that the path $P_n$ maximizes the meeting time $\Tmeet(G) = \Jmax(G)/2|E|$.
  Theorem \ref{thm:meet-path-star} provides the formula for $\Tmeet(P_n)$.
\end{proof}

\section{Conclusion}

We have characterized the extremal values of $\Tmeet(G) = \max_{w \in V} \sum_{v \in V} \pi_v H(v, w)$ for graphs $G \in \Tnd$ on $n$ vertices with diameter $d$. The maximizing tree is the broom graph $B_{n,d}$. For $d=3$, the minimizing structure is the balanced double broom $D_{n,3}$. For $d > 3$, the minimizer depends on the relative parities of $n$ and $d$. When $n$ and $d$ have opposite parities, the minimizing tree is the balanced double broom $D_{n,d}$.  When $n$ and $d$ have the same parity, the minimizer is the balanced near double broom $D_{n,d}'$.

Looking to further work, it would be interesting to characterize the \emph{best meeting time} 
$$
\Tbestmeet(G) = \min_{w \in V} H(\pi, w) = \min_{w \in V} \sum_{v \in V} \pi_v H(v, w)
$$ 
for graphs in $\Tnd$.
More generally, there are other mixing measures besides Kemeny's constant and the meeting time. For example, Beveridge et al.~\cite{BHOV} showed that the balanced double broom $D_{n,d}$ achieves the maximum \emph{exact mixing time} $\max_{v \in V} H(v,\pi)$, where $H(v, \pi) = \max_{w \in V} H(w,v) - H(\pi,v)$. So characterizing the \emph{best mixing time} $\min_{v \in V} H(v,\pi)$ for trees in $\Tnd$ is an open question. For another example,  the formula for the \emph{mean first passage time}
$ \frac{1}{n(n-1)} \sum_{v \in V} \sum_{w \in V} H(v,w)$ is similar to Kemeny's constant. Perhaps techniques from \cite{MaWang2020} could be useful in determining the best and worst trees for this quantity. 


\bibliographystyle{plain}
\bibliography{mybib}

\end{document}